\documentclass[12pt]{amsart}
\usepackage{amsmath,amssymb,amsthm}
\usepackage{setspace}
\usepackage{moreverb}
\usepackage{mathrsfs}
\usepackage{url}
\usepackage[all]{xy}
\usepackage{lscape}
\usepackage{tikz}
\usetikzlibrary{arrows}
\usepackage{color}

\newcommand{\ignore}[1]{}

\newtheorem{theorem}{Theorem}[section]

\newtheorem{corollary}[theorem]{Corollary}
\newtheorem{proposition}[theorem]{Proposition}

\theoremstyle{definition}
\newtheorem{definition}[theorem]{Definition}
\newtheorem{example}[theorem]{Example}

\theoremstyle{remark}
\newtheorem{remark}[theorem]{Remark}

\numberwithin{equation}{section}

\input xy
\xyoption{all}

\newcommand{\cx}{\operatorname{cx}}

\def\calF{{\mathcal F}}

\def\NN{{\mathbb N}}
\def\QQ{{\mathbb Q}}

\def\frakm{{\mathfrak m}}
\def\cm{\divideontimes}

\definecolor{grey}{rgb}{0.75,0.75,0.75}
\definecolor{orange}{rgb}{1.0,0.5,0.5}
\definecolor{brown}{rgb}{0.5,0.25,0.0}
\definecolor{pink}{rgb}{1.0,0.5,0.5}

\newcommand{\rank}{\mbox{\rm rank} }

\newcommand{\End}{\mathrm{End}}

\newcommand{\Hom}{\mbox{\rm{Hom}} }

\newcommand{\fM}{{\mathfrak m}}

\newcommand{\cF}{{\mathcal F}}
\newcommand{\cG}{{\mathcal G}}

\newcommand{\lra}{{\longrightarrow}}

\begin{document}

\title[Generating functions associated to Frobenius algebras]{Generating functions associated to Frobenius algebras}

\author[J. \`Alvarez Montaner]{Josep \`Alvarez Montaner}

\address{Departament de Matem\`atiques\\
Universitat Polit\`ecnica de Catalunya\\ Av. Diagonal 647, Barcelona
08028, Spain} \email{Josep.Alvarez@upc.edu}

\thanks{Partially supported by Generalitat de Catalunya 2017 SGR-932 project and Ministerio de Econom\'ia y Competitividad
MTM2015-69135-P. He is also with the Barcelona Graduate School of Mathematics (BGSMath).}



\begin{abstract}
We introduce a generating function associated to the homogeneous generators
of a graded algebra that measures how far is this algebra from being finitely
generated. For the case of some algebras of Frobenius endomorphisms we 
describe this generating function explicitly as a rational function.

\end{abstract}

\maketitle

\section{Introduction}

%

Let $R$ be a commutative Noetherian ring and let $A= \oplus_{e\geq 0} A_e$ be a $\mathbb{N}$-graded ring, such that $A_0=R$ 
so it has a natural structure as $R$-algebra. 
Although $A$ may not necessarily be commutative,  
we will assume for simplicity that $A$ is a left 
skew $R$-algebra, that is $aR \subseteq Ra$ for all homogeneous elements $a \in A$.

\vskip 2mm

In the case that $A$ is not a finitely generated $R$-algebra it is natural to ask how far is this algebra
from being finitely generated by measuring the number of generators that we have at each degree.
Under these premises,  F.~Enescu and Y.~Yao introduced in \cite{EY16} the {\it complexity sequence} of $A$, where the 
integer components $c_e$ of this sequence  measure the number of homogeneous generators of degree $e$ that can 
not be obtained from homogeneous elements of lower degree. More precisely:

\begin{definition}
Let $A= \oplus_{e\geq 0} A_e$ be a $\mathbb{N}$-graded ring. 
Set  $G_{-1} = A_0$ and let $G_e(A)= G_e$ be the subring of $A$ generated by the elements of
 degree $ \leq e$. Let $c_e:=c_e(A)$ be the number of homogeneous generators of  $A_e/(G_{e-1}(A))_e$
 over $A_0$. Then, the {\it complexity sequence} of $A$ is $\{c_e(A)\}_{e\geq 0}$.
\end{definition}

From now on, we will only consider the case of 
{\it degree-wise finitely generated} $\mathbb{N}$-graded algebras ensuring that
 $c_e < + \infty$ for all $e$.  The asymptotic behaviour of this sequence allowed F.~Enescu and Y.~Yao to introduce a 
 new invariant, the {\it complexity} of $A$.
 
 \begin{definition}
 Let $\{c_e(A)\}_{e\geq 0}$ be the complexity sequence of a  $\mathbb{N}$-graded $R$-algebra $A$.
 The {\it complexity} of $A$ is $$\cx(A):=\inf \{ \lambda \in \mathbb{R}_{>0} \hskip 1mm | \hskip 1mm  c_e(A)=O(\lambda^e) \}.$$ If there is no such an $\lambda$
 we say that $\cx(A)= \infty$.
 \end{definition}

\vskip 2mm

A motivating example is when $R$ is a local complete commutative ring of positive characteristic $p>0$ and $A=\calF(E_R)$ is the
Frobenius algebra associated to the injective hull of the residue field, that we denote as $E_R$. 
This is a non-commutative $\mathbb{N}$-graded $R$-algebra introduced by G.~Lyubeznik and K.~E.~Smith in \cite{LS01}
that collects all possible Frobenius actions on $E_R$. The study of these algebras has its roots in the theory of tight closure
introduced by M.~Hochster and C.~Huneke \cite{HH90}. The dual notion of Cartier algebra (see \cite{Sch11}, \cite{Bli13}) 
plays a prominent role in the theory of singularities in positive characteristic. 

\vskip 2mm

For the case of Frobenius algebras, F.~Enescu and Y.~Yao 
coined  in \cite{EY16} the notion of {\it Frobenius complexity}. Its interest
comes from the fact that, for some particular examples, the limit of this 
invariant as $p \rightarrow \infty$ exists so it may be interpreted as an invariant
of $R$ in characteristic zero. 

 \begin{definition}
Let $R$ be a local complete commutative ring of positive characteristic $p>0$.
The {\it Frobenius complexity} of $R$ is $\cx_F(R):=\log_p \cx(\calF(E_R))$.
 \end{definition}
 
 M.~Katzman gave  in \cite{Kat10} the first example of non-finitely generated
 Frobenius algebra $\calF(E_R)$. Unfortunately, we may not find many examples in the literature where 
 non-finitely generated  Frobenius algebras are explicitly described so one may extract the complexity sequence. 
 Among these scarce sources we would like to mention the cases of  Stanley-Reisner rings \cite{ABZ12}, Veronese subrings \cite{KSSZ}
 or $2\times 2$ minors of a $2 \times 3$ generic matrix  \cite{KSSZ}. For the case of $2\times 2$ minors of a $n \times m$ generic matrix,
 F.~Enescu and Y.~Yao \cite{EY16, EY15} gave an indirect approach to compute the complexity sequence.
 Actually, due to the asymptotic nature of these invariants, one may bound or even compute the Frobenius complexity
 by comparing with some known cases (see \cite{EY16, EY15}). J.~Page \cite{Pag17} also took this indirect approach to compute  
 the Frobenius complexity of the so-called Hibi rings.

\vskip 2mm

In this work, instead of studying the asymptotic behaviour of the complexity sequence, we are more concerned
on structural properties of this sequence. In Section \ref{Sec:2} we collect the terms of this sequence in a series 
$$\cG_A(T)=\sum_{e\geq 0} c_e T^e$$ that we denote as the {\it generating function} of $A$. In the case that 
$\cG_A(T)$ is a rational function we obtain an explicit linear recurrence for the coefficients $c_e$. 
This is a very strong condition  that implies that the complexity of $A$ is finite since we can read
this invariant from the set of poles of the generating function (see Theorem \ref{pole}). 

\vskip 2mm

In Section \ref{Sec:3} we recall the basics on Frobenius algebras which is the main example we are interested on. 
Our interest comes from the fact that, for the few examples that we may find in the literature, the corresponding
generating function is rational. These computations are developed in Section \ref{Sec:4}. We pay special attention
to the case of determinantal rings treated by F.~Enescu and Y.~Yao. The main result of this section is Theorem \ref{recurrence}
where we describe the linear recurrence that the complexity sequence satisfies in this case. 

\vskip 2mm

All the computations performed in this paper have been developed using {\tt MATLAB} \cite{mat}. 

\vskip 2mm

{\it Aknowledgements:} We greatly appreciate the referee for his/her comments.

\section{Generating function} \label{Sec:2}

Let $R$ be a commutative Noetherian ring and let $A= \oplus_{e\geq 0} A_e$ be a 
(non-necessarily commutative) $\mathbb{N}$-graded ring.
Associated to the complexity sequence of $A$ we may consider the following series:

\begin{definition} \label{poincare}
Let $\{c_e(A)\}_{e\geq 0}$ be the complexity sequence of a  $\mathbb{N}$-graded algebra $A$.
 We define {\it the generating function} of $A$ as
$$\cG_A(T)=\sum_{e\geq 0} c_e T^e$$

\end{definition}

\begin{remark}
The generating function encodes the same information as the $\mathscr Z$-transform of the complexity sequence $\{c_e(A)\}_{e\geq 0} $.
Namely,  $\cG_A(T)=\mathscr Z[c_e](\frac{1}{T})$ where 
$$\mathscr Z[c_e](Z)=\sum_{e\geq 0} \frac{c_e}{Z^e}.$$ 
This alternative approach will be useful in Section \ref{determinantal} where we are going to use some elementary properties of the $\mathscr Z$-transform without further comment. 
There is a vast bibliography on the $\mathscr Z$-transform, especially in mathematical engineering, so we refer to any textbook  such as \cite{Ela} for more insight.

\end{remark}

\begin{remark}
We may define another { generating function} of $A$ as
$$\widehat{\cG}_A(T)=\sum_{e\geq 0} a_e T^e,$$
where  $a_e:=a_e(A)$ is the number of homogeneous generators of  $A_e$ over $A_0$. 
The motivation behind considering $\cG_A(T)$ instead of $\widehat{\cG}_A(T)$ in this work 
is because of its relation with the complexity of $A$.
\end{remark}

The natural question that arise whenever we have a series as the one defined by the generating function is whether 
it is a rational function. This property provides a strong structural property of the complexity sequence due to the 
following well-known result for which we present a sketch of the proof for the sake of completeness.

\begin{proposition} \label{rational}
Let $\{c_e(A)\}_{e\geq 0}$ be the complexity sequence of a  $\mathbb{N}$-graded algebra $A$. 
Let $Q(T)= 1 + a_1 T+ a_2 T^2 + \cdots +  a_{k-1} T^{k-1} + a_{k} T^k$  be a polynomial of degree $k$ in one variable $T$. Then the following are equivalent:

\vskip 2mm

\begin{itemize}
 \item[i)] There exists an integer $e_0\geq 0$ such that, for all $e\geq e_0$, the coefficients of the complexity sequence are given by a linear recurrence relation
$$c_{e+k} + a_1 c_{e+k-1}+ a_2 c_{e+k-2} + \cdots +  a_{k-1} c_{e+1} + a_{k} c_{e} = 0.$$ 

 \item[ii)]  The generating function is a rational function 
 $$\cG_A(T)= c_0+ c_1 T + \cdots + T^{e_0} \sum_{e\geq e_0} c_e T^{e-e_0}= c_0+ c_1 T + \cdots + T^{e_0} \frac{P(T)}{Q(T)},$$
 with  $\deg P(T)<k$.
\end{itemize}

\end{proposition}

\begin{proof}
Without loss of generality we may assume $e_0=0$.
We have:

\begin{align*}
Q(T) \sum_{e\geq 0} c_e T^e &=(1 + a_1 T+  a_2 T^2 + \cdots +  a_{k-1} T^{k-1} + a_{k} T^k) \sum_{e\geq 0} c_e T^e =  \\
&  = c_0 + (c_1+a_1c_0)T + \cdots + (c_{k-1}+a_1c_{k-2}+\cdots + a_{k-1}c_0) T^{k-1} + \\
& + (c_{k}+a_1c_{k-1}+\cdots + a_{k}c_0) T^{k} + (c_{k+1}+a_1c_{k}+\cdots + a_{k}c_1) T^{k+1} + \cdots
\end{align*}
\vskip 2mm
Therefore, 
$$c_{e+k} + a_1 c_{e+k-1}+ a_2 c_{k-2} + \cdots +  a_{k-1} c_{e+1} + a_{k} c_{e} = 0,$$  for all $e\geq 0$ if and only if  $\cG_A(T)= \frac{P(T)}{Q(T)},$
with $$P(T)=c_0 + (c_1+a_1c_0)T + \cdots + (c_{k-1}+a_1c_{k-2}+\cdots + a_{k-1}c_0) T^{k-1}.$$
\end{proof}

\begin{theorem} \label{pole}
Assume that the  generating function  $\cG_A(T)$ of a  $\mathbb{N}$-graded algebra $A$ is a rational function. Then
the complexity of $A$ is $\cx(A)= \frac{1}{|\lambda|},$ where $\lambda$ is  a pole with minimal absolute value. 

\end{theorem}

\begin{proof}
Once again, assume for simplicity, that  $$\cG_A(T)= \frac{P(T)}{Q(T)},$$
 with  $Q(T)= 1 + a_1 T+ a_2 T^2 + \cdots +  a_{k-1} T^{k-1} + a_{k} T^k$ and $\deg P(T)<k$. Thus the recurrence relation is
 $$c_{k} + a_1 c_{k-1}+ a_2 c_{k-2} + \cdots +  a_{k-1} c_{1} + a_{k} c_{0} = 0.$$
In order to solve this linear difference equation, let  $$q(T)=T^k + a_1 T^{k-1}+ a_2 T^{k-2} + \cdots +  a_{k-1} T + a_{k} $$ be its characteristic polynomial. We have $q(T)=Q(\frac{1}{T})$, so the inverses of the roots of $q(T)$ correspond to the poles of $\cG_A(T)$.
Assume that we have a factorization $$q(T)=(T-\lambda_1)^{m_1}(T-\overline{\lambda_1})^{m_1} \cdots (T-\lambda_r)^{m_r}(T-\overline{\lambda_r})^{m_r}(T-\lambda_{r+1})^{m_{r+1}} \cdots (T-\lambda_s)^{m_s}, $$ with $\lambda_j = |\lambda_j| e^{i\Theta_j}\in \mathbb{C}$ for $j=1,\dots,r$ and $\lambda_j \in \mathbb{R}$ for $j=r+1,\dots, s$.
Then the elements $c_e(A)$ of the complexity sequence are linear combinations of the form

\begin{align*}
c_e(A) & =|\lambda_1|^e \left[ (\alpha_{1,1}  + \alpha_{1,2} e  + \cdots + \alpha_{1,m_1} e^{m_1-1} ) \cos \Theta_1 +  (\alpha'_{1,1}  + \alpha'_{1,2} e  + \cdots + \alpha'_{1,m_1} e^{m_1-1} ) \sin \Theta_1  \right] \\
&  \hskip 8cm \vdots  \\
&+ |\lambda_r|^e \left[ (\alpha_{r,1}  + \alpha_{r,2} e  + \cdots + \alpha_{r,m_r} e^{m_r-1} ) \cos \Theta_r +  (\alpha'_{r,1}  + \alpha'_{r,2} e  + \cdots + \alpha'_{r,m_r} e^{m_r-1} ) \sin \Theta_r  \right]\\
& + \lambda_{r+1}^e \left[ \alpha_{r+1,1}   + \cdots + \alpha_{r+1,m_{r+1}} e^{m_{r+1}-1}\right] + \cdots  
+ \lambda_{s}^e \left[ \alpha_{s,1}  +  \cdots + \alpha_{s,m_{s}} e^{m_{s}-1}\right]. 
\end{align*}

\vskip 2mm

\noindent Notice that each  term  has order $O(|\lambda_i|^e)$ and thus  the order of $c_e$ corresponds to the roots of $q(T)$ with maximal absolute value.  These roots are the inverse of  poles  of $\cG_A(T)$ with minimal absolute value and thus $\cx(A)= \frac{1}{|\lambda|} $ where $\lambda$ is such a pole.

\end{proof}

\begin{remark}
The coefficients $\alpha_{i,j}$ and $\alpha'_{i,j}$  in the linear combination described above 
are uniquely determined imposing initial conditions $c_0, \dots, c_{k-1}$. Therefore, the order of the complexity sequence depends on these initial conditions.
%
%
\end{remark}

\begin{remark}
In all the examples that we will consider, the linear difference equation has a {\it dominant characteristic eigenvalue}  
so their generating function $\cG_A(T)$ has a unique simple pole with minimum absolute value. 

\end{remark}

\begin{corollary} \label{linear_cx}
Let $\{c_e(A)\}_{e\geq 0}$ be the complexity sequence of a  $\mathbb{N}$-graded algebra $A$ satisfying the linear recurrence relation
$$c_{e+k} + a_1 c_{e+k-1}+ a_2 c_{e+k-2} + \cdots +  a_{k-1} c_{e+1} + a_{k} c_{e} = 0.$$ 
Then, the complexity of $A$
is $\cx(A)=|\lambda|$, if $\lambda$ is a  root of the characteristic polynomial $$q(T)=T^{k} + a_1 T^{k-1}+ a_2 T^{k-2} + \cdots +  a_{k-1} T + a_{k} $$  with maximum absolute value.

\end{corollary}

\section{Frobenius algebras} \label{Sec:3}

In this section we will introduce the basics on the theory of Frobenius algebras
that we are going to use in the rest of this paper. We point out that one could also 
consider the dual notion of Cartier algebras so the interested reader should feel free
to follow her/his own preference.

\vskip 2mm
 
Let $R$ be a commutative ring of characteristic $p>0$ and $M$ an $R$-module.
We use the $e$-th iterated Frobenius map $F^e:R\lra R$ to define a new $R$-module
structure on $M$ given by $rm:=r^{p^e}m$ for all $r\in R$ and $m\in M$. One denotes this $R$-module 
as $F^e_{\ast}M$.  Indeed, one may use this to define the   $e$-th  Frobenius functor from the category
of left $R$-modules to itself. In \cite{LS01}, G.~Lyubeznik and K.~E.~Smith  introduced the following
ring in their study of the localization problem in tight closure.

\begin{definition}
The ring of Frobenius operators on $M$ is the $\mathbb{N}$-graded, associative, not necessarily commutative  ring
$$\cF(M):=\bigoplus_{e\geq 0}\cF^e(M),$$ where $\cF^e(M):=\Hom_R(M,F_{\ast}^eM).$
\end{definition}

\vskip 2mm

The pieces $\cF^e(M)$ can be identified with the set of all \textit{$p^{e}$-linear maps}. That is, the set of all
additive maps $\varphi_e: M\lra M$ satisfying
$\varphi_{e}(r m)=r^{p^e}\varphi_{e}(m) $ for all $r\in R$, $m \in M$.
Composing a $p^{e}$-linear map  and a $p^{e'}$-linear map
 we get a $p^{(e+e')}$-linear map so we get a natural ring structure for $\cF(M)$.
 Each $\cF^e(M)$ is a left module over $\cF^0(M):=\End_R(M)$.
 
 \vskip 2mm

When $(R,\fM, k)$ is a local complete local ring and $E_R:=E_R(k)$ is the injective hull of the residue field
we notice that $\cF(E_R)$ is an $R$-algebra due to the fact that
\begin{equation*}
\cF^0(E_R)=\Hom_R(E_R,E_R)\cong R.
\end{equation*}


 \vskip 2mm

Two important instances where the Frobenius algebra can be described more explicitly are presented in the sequel.
 
 \vskip 2mm

 \subsection{The case of completed $k$-algebras of finite type}

 Let $S=k[[x_1,...,x_n]]$ be the formal power series ring in $n$
variables over a field $k$ of characteristic $p>0$. Let $I\subseteq
S$ be any ideal and  $E_R$ be the injective hull of the residue field of $R:=S/I$.  For this module we have a nice description of the corresponding Frobenius algebra using a result of R.~Fedder \cite{Fed83} (see  also \cite{Bli01}, \cite{Kat08} for more insight). Namely, there exists a natural Frobenius action $F$ from $E_R$ onto itself such that for each $e\geq 0$, any $p^e$-linear map from $E_R$ onto itself is uniquely of the form $gF^e$, where $g$ is an element of $(I^{[p^e]}:_S I)/ I^{[p^e]}$. So there exists an isomorphism of $R$-modules $$\cF^e(E_R) \cong  (I^{[p^e]}:_S I)/ I^{[p^e]}$$ that can be extended in a natural way to an isomorphism of $R$-algebras $$\cF(E_R)\cong \bigoplus_{e\geq 0} (I^{[p^e]}:_S I)/ I^{[p^e]}.$$


%
%

\vskip 2mm

\subsection{T-construction}

Let $(R,\fM,k)$ be a normal complete local ring  of characteristic $p>0$ and $E_R$ be the injective hull of the residue field.
Let $\omega_R$ be a canonical module of $R$. Then, the \emph{anticanonical cover} of $R$ is the $\NN$-graded ring
\[
\mathscr R\ =\ \bigoplus_{n\ge0}\omega_R^{(-n)}
\]
where $\omega_R^{(-n)}$ denotes the $n$-th symbolic power of the divisorial ideal $\omega_R^{(-1)}$.
M.~Katzman, K.~Schwede, A.~Singh and W.~Zhang showed in \cite{KSSZ} that the Frobenius algebra $\cF(E_R)$ is isomorphic to a
subgroup of the anticanonical cover of $R$ with a twisted multiplication that we are going to describe next

\begin{definition}
\label{defn:cm}

Let $\mathscr R$ be an $\NN$-graded commutative ring  of characteristic $p>0$.
Associated to this ring, we consider
\[
T(\mathscr R)\ =\ \bigoplus_{e\ge0}\mathscr R_{p^e-1}\,,
\]
with a twisted multiplication $\cm$ on $T(\mathscr R)$ given by
\[
a\cm b\ =\ ab^{p^e}\qquad\text{ for }a\in{T(\mathscr R)}_e\text{ and }b\in{T(\mathscr R)}_{e'}\,.
\]
\end{definition}

Using this construction, we get the following interpretation of the Frobenius algebra 
(see \cite[Theorem 3.3]{KSSZ} for details).

\begin{theorem}
\label{theorem:main}
Let $(R,\frakm)$ be a normal complete local ring of characteristic $p>0$.  
Then there exists a graded isomorphism $$\calF(E_R)\cong T(\mathscr R)=\bigoplus_{e\ge0}\omega_R^{(1-p^e)}$$
\end{theorem}

It follows that $\calF(E_R)$ is principal whenever $R$ is Gorenstein.
In the case that $R$ is $\QQ$-Gorenstein, combining \cite[Proposition 4.1]{KSSZ} and \cite[Theorem 4.15]{EY15},
we have

\begin{proposition}
\label{prop:q:gor}
Let $(R,\frakm, k)$ be a  normal complete $\QQ$-Gorenstein local ring of characteristic $p>0$.  Then 

\begin{enumerate}

\item $\calF(E_R)$ is a finitely generated $R$-algebra if and only if $p$ is relatively prime with the index of $R$.

\item   $\calF(E_R)$ is principal if and only if the index of $R$ divides $p-1$.

\end{enumerate}

\end{proposition}

\section{Generating functions of some Frobenius algebras} \label{Sec:4}

As we mentioned in the introduction, there are not so many explicit examples 
of Frobenius algebras that we may find in the literature. In this section we will present them
and we will show that the generating function is a rational function, or equivalently, the complexity
sequence satisfies a linear recurrence.

 \subsection{Stanley-Reisner rings} \label{monomial}
  The first example of a non finitely generated Frobenius algebra $\cF(E_R)$
 was given by M.~Katzman in \cite{Kat10} where he considered the (completed)
 Stanley-Reisner ring  $R=k[[x,y,z]]/(xy,yz)$.

 \vskip 2mm
More generally, let $S =k[[x_1,\dots,x_m]]$ be the formal power series ring with
coefficients over a field of positive characteristic and let $I$ be a squarefree monomial ideal.  For simplicity, we may just consider ideals involving all the variables.
A complete description of the Frobenius algebra $\cF(E_R)$ associated to the Stanley-Reisner
ring $R=S/I$ is given in \cite{ABZ12}. Their approach is by means of an explicit
computation of  the colon  ideal $(I^{[p^e]}:_S I)$. It turns out
 that there are only two possible cases (see \cite[Section 3.1]{ABZ12}
 for details)

 \begin{enumerate}

\item $(I^{[p^e]}:_S I)= I^{[p^e]} + (x_1\cdots x_m)^{p^e-1}$.

\vskip 2mm

\item $(I^{[p^e]}:_S I)= I^{[p^e]} + J_{p^e} + (x_1\cdots x_m)^{p^e-1}$.
\end{enumerate}

\vskip 2mm

\noindent where  $J_{p^e}$ is generated by  monomials ${\bf x}^{\gamma}:=x_1^{c_1}\cdots x_m^{c_m}$ satisfying
$c_i\in \{0,p^e-1,p^e\}$.
These monomials are not contained in $I^{[p^e]} + (x_1\cdots x_m)^{p^e-1}$ whenever
$c_i=p^e$, $c_j=p^e-1$, $c_k=0$ for $1\leq i,j,k \leq m$.

\vskip 2mm

As an immediate consequence we have the following result.

\begin{theorem} \cite[Theorem 3.5]{ABZ12}
The Frobenius algebra $\cF(E_R)$ associated to a Stanley-Reisner ring $R$ is either
principally generated or infinitely generated.
\end{theorem}

More interestingly, the formula obtained for $(I^{[p^e]}:_S I)$ is exactly the same $\forall e$,
so one only has to compute the first graded piece $\cF^1(E_R)\cong (I^{[p]}:_S I)$ in order
to describe the whole Frobenius algebra  $\cF(E_R)$. When $\cF(E_R)$ is principally generated it is generated by 
$(x_1\cdots x_m)^{p-1}$. On the other hand,
when it is infinitely generated, we have that  $\cF^e(E_R)$ has $\mu$ minimal monomial generators
corresponding to the minimal generators of  $J_{p^e}$ plus the generator $(x_1\cdots x_m)^{p^e-1}$.
A.~F.~Boix and S.~Zarzuela \cite[Theorem 2.15]{BZ16} gave a nice interpretation of the monomial generators in terms of the  {\it maximal free pairs} of the simplicial complex associated to the Stanley-Reisner ring $R$.

\vskip 2mm

A full description of the  complexity sequence is still an open question.  A.~F.~Boix and S.~Zarzuela proved in \cite[Theorem 3.8]{BZ16} that the generators coming from monomials ${\bf x}^{\gamma} \in J_{p^e}$ with minimal support\footnote{The support of a monomial
${\bf x}^{\gamma}:=x_1^{c_1}\cdots x_n^{c_n}$ is ${\rm supp}({\bf x}^{\gamma}):=\{ i \in \{1,\dots , n\} \hskip 2mm | \hskip 2mm c_i \neq 0 \}$.} are new,  that is they cannot be  obtained from the $R$-algebra generated by $\cF^{0}(E_R), \dots , \cF^{e-1}(E_R)$.
Then, they conclude that 
each piece $\cF^e(E_R)$ adds at most $\mu$ new generators. 
In the case that all the monomials in $J_{p^e}$ have minimal support we have that the complexity sequence is
$$\{c_e\}_{e\geq 0}= \{ 1, \mu+1, \mu , \mu , \mu , \dots \},$$ so we have the linear recurrence $c_{e+1}-c_e=0$ for all $e\geq 2$.
In particular, its generating function is
 $$\cG_{\cF(E_R)}(T)= 1 + (\mu +1) T + \sum_{e\geq 2} \mu T^e = 1 + (\mu +1) T + T^2 \frac{\mu }{1-T}= \frac{1+\mu T - T^2}{1-T}.$$
 
\begin{example}
Let $I=(x,y)\cap (z,w)$ be a squarefree monomial ideal in $S =k[[x,y,z,w]]$. Then we have 
$$\frac{(I^{[p^e]}:_S I)}{ I^{[p^e]}}= (x^{p^e-1}y^{p^e-1}z^{p^e},x^{p^e}z^{p^e-1}w^{p^e-1},y^{p^e}z^{p^e-1}w^{p^e-1},x^{p^e-1}y^{p^e-1}w^{p^e})+ (xyzw)^{p^e-1}$$
and the generating function is
 $$\cG_{\cF(E_R)}(T)= \frac{1+ 4 T - T^2}{1-T}.$$
\end{example}

\vskip 2mm

\subsection{Veronese subrings} \label{veronese}
Let $R=k[x_1,\dots , x_m]$ be the polynomial ring over a field of characteristic $p>0$
that we consider as a  $\mathbb{N}$-graded ring $R= \bigoplus_{k\geq 0} R_k$. Given an integer 
$r\geq 1$, the $r$-th Veronese  subring of $R$ is $$V_r(R)= \bigoplus_{k \geq 0} R_{rk}.$$

M.~Katzman, K.~Schwede, A.~Singh and W.~Zhang gave a complete description of the Frobenius
algebra $\calF(E_R)$ in \cite[Proposition 4.3]{KSSZ} using the T-construction.

\begin{proposition}
 Under the previous assumptions, $\calF^e(E_R)$ is the left $R$-module generated by the elements
\[
\frac{1}{x_1^{\alpha_1}\cdots x_d^{\alpha_d}}F^e\,,
\]
where $F^e$ is the $e$-th Frobenius  map, $\alpha_k\le p^e-1$ for each $k$, and $\sum\alpha_k\equiv0\mod n$.
 
\end{proposition}

In \cite[Example 4.5]{KSSZ} they consider the case where $m=2$ and $r=3$. The Frobenius algebra
$\calF(E_R)$ is infinitely generated when the characteristic of the field $k$ is $p=3$. A straightforward
computation shows that for $m=2$ and $r=p$, the 
complexity sequence is $\{c_e\}_{e\geq 0}=\{1,p-1,p-1,p-1,\dots\}$ so the generating function is
 $$\cG_{\cF(E_R)}(T)= 1 +\sum_{e\geq 1} (p-1) T^e = 1 +   T \frac{(p-1) }{1-T}= \frac{1+ (p-2)T }{1-T}.$$
The complexity of the Frobenius algebra is $\cx(\cF(E_R))=1$  and the  Frobenius complexity is $\cx_F(R)=0$.

\subsection{Determinantal rings} \label{determinantal}

Let $X$ be an $n\times m$ generic matrix with  $m>n \geq 2$ and $I:=I_2(X)$ the ideal generated by the $2 \times 2$
minors of $X$.  In this subsection we will compute the generating function of the 
Frobenius algebra  $\calF(E_R)$, where $R$ is the completion of the determinantal ring $k[X]/I$.

\vskip 2mm

A precise description of the Frobenius algebra for the case of a $2\times 3$ generic matrix
using Fedder's approach was given in \cite{KSSZ} so, from this construction one may extract
its complexity sequence. However, F.~Enescu and Y.~Yao took a different approach in 
\cite{EY16} and \cite{EY15} to describe the complexity sequence of $\calF(E_R)$ in the general case
of a $n \times m$ generic matrix. 
Actually, they proved in \cite[Theorem 1.20]{EY15} that the complexity sequence of $\calF(E_R)$ coincides with the complexity sequence 
of the $\mathbb{N}$-graded ring $A=T( V_{m-n} (k[x_1,\dots , x_m]))$, which is the twisted ring associated to the
$(m-n)$-th Veronese subring of a polynomial ring $k[x_1,\dots , x_m]$ in $m$ variables.  

\vskip 2mm

Next we are going to describe this complexity sequence in the general framework considered
in \cite{EY16}. Namely, they consider the twisted ring $T( V_r (R[x_1,\dots , x_m]))$,
where $R$ is any commutative ring $R$ of characteristic $p>0$ not only a field.
In  \cite[Proposition 3.1]{EY16}  they give a precise description but, for our purposes, we will consider the interpretation
they provide in \cite[Discussion 3.2]{EY15} that we briefly describe.

\vskip 2mm

Set
$$M_{p,m}(k):= \rank_R \left( R[x_1,\dots,x_m]/(x_1^p,\dots,x_m^p)\right)_k.$$
These positive integers can be read off as the coefficients of the following polynomial
$$\sum_{k=0}^{m(p-1)} M_{p,m}(k) T^k = (1+T+ \cdots + T^{p-1})^m$$
In particular $M_{p,m}(k)=0$ for $k<0$ and $k>m(p-1)$. Moreover they satisfy the
symmetric property $M_{p,m}(k)=M_{p,m}(m(p-1)-k)=0$.

\vskip 2mm

Now we construct the $(m-r-1) \times (m-r-1)$ matrix

{\footnotesize
$$ U:=  \left(
\begin{array}{cccc}
M_{p,m}(r(p-1)+p-1)& \cdots & M_{p,m}(r(p-1)+p-(m-r-1))\\
\vdots &  \ddots & \vdots \\
M_{p,m}(r(p-1)+(m-r-1)p-1)& \cdots & M_{p,m}(r(p-1)+(m-r-1)p-(m-r-1))\\
\end{array}
\right)$$
}

\noindent that is, $ U= \left( u_{ij} \right) $
with $u_{ij}:=M_{p,m}(r(p-1)+ip-j)$.

\vskip 2mm

Consider the discrete dynamical system $X_{e}=U X_{e-1}$, or equivalently $X_e = U^e X_0,$ with the initial conditions

{\small $$X_0=\left( M_{p,m}(r(p-1)+p), M_{p,m}(r(p-1)+2p), \dots , M_{p,m}(r(p-1)+(m-r-1)p)\right)^{\top}.$$}


The complexity sequence of $T( V_r (R[x_1,\dots , x_m]))$  for any positive integers $r,m$ such that $r+1<m$ is:

\begin{itemize}
\item[$\cdot$] $c_0=1.$

\item[$\cdot$] $c_1=\rank_R \left(R[x_1,\dots , x_m]\right)_{r(p-1)}= {r(p-1)+(m-1) \choose m-1}.$

\item[$\cdot$] $c_e= Y \cdot X_{e-2},$ for $e\geq 2$, where:
\end{itemize}

{\footnotesize $$Y=\left( {r(p-1)-1+m-1 \choose m-1}, {r(p-1)-2+m-1 \choose m-1}, \dots , {r(p-1)-(m-r-1)+m-1 \choose m-1}\right).$$ }

\vskip 2mm

The main result of this subsection is the following result that describes the linear recurrence of this complexity sequence.

\begin{theorem} \label{recurrence}
Let $m(T)= T^s+a_1 T^{s-1} + \cdots + a_s $ be the minimal polynomial of the matrix $U$ associated to the twisted 
Veronese subring $A=T( V_r (R[x_1,\dots , x_m]))$. Then, the complexity sequence  of  $\{c_e(A)\}_{e\geq 0}$ 
satisfies the  linear recurrence $$ c_{e+s}+a_1 c_{e+s-1} + \cdots + a_s c_e=0, $$ for all $e\geq 2$.
\end{theorem}

\begin{proof}
We start with the case where  the matrix $U$ diagonalizes.  
Let $v_1,\dots , v_{m-r-1}$ be the set of eigenvectors of $U$ with eigenvalues $\lambda_1, \dots , \lambda_{s}$ with $s\leq m-r-1$.
In particular, if the characteristic polynomial of $U$ is $p(T)=(T-\lambda_1)^{m_1}\cdots (T-\lambda_s)^{m_s}$, then the minimal polynomial 
 is $m(T)=(T-\lambda_1)\cdots (T-\lambda_{s})$.
Let $S$ be the change basis matrix such that $U=S D S^{-1}$ with $D$ being the diagonal matrix with the set of eigenvalues as entries.
Thus we have:
$$X_e= U^e X_0=S D^e S^{-1} X_0= \alpha_1 \lambda_1^e v_1 + \cdots + \alpha_{m-r-1} \lambda_{s}^e v_{m-r-1},$$
where $S^{-1} X_0=(\alpha_1 , \dots , \alpha_{m-r-1})^\top$.
Therefore, the elements of the complexity sequence are
$$c_e= Y \cdot X_{e-2}= (\alpha_1 \lambda_1^{e-2} ) Y\cdot v_1 + \cdots + (\alpha_{m-r-1} \lambda_{s}^{e-2}) Y\cdot  v_{m-r-1}.$$
Recall that the $\mathcal{Z}$-transform of the sequence $\{\lambda^e\}_{e\geq 0}$ is 
$$\mathscr Z[\lambda^e](Z)= \frac{Z}{Z-\lambda} \hskip 3mm , \hskip 4mm 
\mathscr Z[\lambda^e](\frac{1}{T})=\frac{1}{1-\lambda T}.$$ It follows that
$$\sum_{e\geq 2} c_e T^{e-2}= \frac{\alpha_1 Y\cdot v_1  }{1-\lambda_1 T} + \cdots + \frac{\alpha_{m-r-1} Y\cdot v_{m-r-1}  }{1-\lambda_{s} T} = 
\frac{P(T)}{(1-\lambda_1 T)\cdots (1-\lambda_{s}T)} $$ with $\deg(P (T) )< s $  and therefore
$$\mathcal{G}_A(T)= c_0+c_1 T+ T^2\sum_{e\geq 2} c_e T^{e-2}= c_0+c_1 T+ T^2\frac{P(T)}{(1-\lambda_1 T)\cdots (1-\lambda_{s}T)}$$
and the result follows from Proposition \ref{rational}.

\vskip 2mm

If the matrix $U$ does not diagonalize, we may play the same game with the corresponding Jordan normal form.
Let $m(T)=(T-\lambda_1)^{j_1}\cdots (T-\lambda_{s})^{j_s}$ be the minimal polynomial of $U$ and
$S$  the change basis matrix such that $U=S J S^{-1}$ with $J$ being the normal Jordan form.

\vskip 2mm

To avoid heavy notation, we will start developing carefully the case where $J$ is a  $j \times j$ Jordan block corresponding 
to an eigenvalue $\lambda$. Notice that we are assuming that the characteristic and the minimal polynomial are
both $(T-\lambda)^{j}$. Let $v_1,\dots , v_{j}$ be the set of generalized eigenvectors. Then

\begin{align*}
X_e & =U^e X_0=S J^e S^{-1} X_0 = \alpha_1 [ \lambda^e v_1 +  \binom{e}{1} \lambda^{e-1} v_2+\cdots + 
\binom{e}{j-1} \lambda^{e-(j-1)} v_{j} ] +  \\ & + \alpha_2 [\lambda^e v_2 +  \cdots + 
\binom{e}{j-2} \lambda^{e-(j-2)} v_{j}] + \cdots + \alpha_j [\lambda^e v_j]
\end{align*}

where $S^{-1} X_0=(\alpha_1 , \dots , \alpha_{j})^\top$.
Therefore, we have
\begin{align*}
c_e &= Y \cdot X_{e-2}=   \\ &= (\alpha_1 \lambda^{e-2} ) Y\cdot v_1 + (\alpha_1 \binom{e-2}{1} \lambda^{e-3}) Y\cdot v_2 
+ \cdots + (\alpha_1 \binom{e-2}{j-1} \lambda^{e-2-(j-1)}) Y\cdot v_j +  \\ 
& +  (\alpha_2 \lambda^{e-2} ) Y\cdot v_2 +
+ \cdots + (\alpha_2 \binom{e-2}{j-2} \lambda^{e-2-(j-2)}) Y\cdot v_j +
\cdots + (\alpha_{j}  \lambda^{e-2}) Y\cdot  v_{j}.
\end{align*}

The  $\mathcal{Z}$-transform of the sequence $\{{e \choose i}\lambda^{e-i}\}_{e\geq 0}$ is 
$$\mathscr Z[{e \choose i}\lambda^{e-i}](Z)= \frac{Z}{(Z-\lambda)^{i+1}} \hskip 3mm , \hskip 4mm  
\mathscr Z[{e \choose i}\lambda^{e-i}](\frac{1}{T})=\frac{T^{i}}{(1-\lambda T)^{i+1}}.$$
Thus we have

\begin{align*}
\sum_{e\geq 2} c_e T^{e-2} &=
\frac{\alpha_1 (Y\cdot v_1)  }{1-\lambda T} + \frac{\alpha_{1} (Y\cdot v_{2})  T }{(1-\lambda T)^2}+ \cdots + 
\frac{\alpha_{1} (Y\cdot v_{j})  T^{j-1} }{(1-\lambda T)^j} + \\ &+
\frac{\alpha_2 (Y\cdot v_2)  }{1-\lambda T} +  \cdots + 
\frac{\alpha_{2} (Y\cdot v_{j})  T^{j-2} }{(1-\lambda T)^{j-1}} +  \cdots +  \frac{\alpha_j (Y\cdot v_j)  }{1-\lambda T} = 
\\ &= \frac{P_\lambda(T)}{(1-\lambda T)^j}, 
\end{align*}
where $ P_\lambda (T) $ is a polynomial with $\deg(P_\lambda (T) )< j$.

\vskip 2mm
In the general case we will get a rational function of the form $\frac{P_\lambda(T)}{(1-\lambda T)^j}$ for each Jordan block
of size $j$. Since the maximum size of the Jordan blocks associated to an eigenvalue
is the algebraic multiplicity of this eigenvalue as a root of the minimal polynomial we get 
$$\mathcal{G}_A(T)= c_0+c_1 T+ T^2\sum_{e\geq 2} c_e T^{e-2}= c_0+c_1 T+ T^2\frac{P(T)}{(1-\lambda_1 T)^{j_1}\cdots (1-\lambda_{s}T)^{j_s}}$$
with $\deg(P (T) )< j_1+ \cdots + j_s $ and the result follows from Proposition \ref{rational}

\end{proof}

\begin{remark}
Numerical experimentation with  {\tt MATLAB} \cite{mat}   suggests that all the eigenvalues of the matrix $U$ are different (albeit we may
have complex eigenvalues). Thus the matrix $U$ diagonalizes and the characteristic and minimal polynomial coincide.  
In particular, the linear recurrence that the complexity sequence satisfies is given by the characteristic polynomial.
\end{remark}

\begin{example} Consider the ring $A=T( V_1 (R[x_1,\dots , x_{18}]))$, where $R$ is a ring of characteristic $p=31$.
Then, the set of eigenvalues of the corresponding matrix $U$ is
 
 \begin{verbatim}
  Z =  1.0e+25 *

  2.255011677416268 + 0.000000000000000i
  0.072742312174691 + 0.000000000000000i
  0.002346526199128 + 0.000000000000000i
  0.000075694393453 + 0.000000000000000i
  0.000002441754572 + 0.000000000000000i
  0.000000078766244 + 0.000000000000000i
  0.000000002540832 + 0.000000000000000i
  0.000000000081958 + 0.000000000000000i
  0.000000000002643 + 0.000000000000000i
  0.000000000000085 + 0.000000000000000i
  0.000000000000003 + 0.000000000000000i
  0.000000000000000 + 0.000000000000000i
 -0.000000000000000 + 0.000000000000000i
 -0.000000000000000 - 0.000000000000000i
  0.000000000000000 + 0.000000000000000i
  0.000000000000000 + 0.000000000000000i
   \end{verbatim}
   
 If we take a close look we can check that there exist complex eigenvalues:

   \begin{verbatim}
>> Z(13)   ans = -2.634600279958723e+08 + 2.351626510499541e+08i
>> Z(14)   ans = -2.634600279958723e+08 - 2.351626510499541e+08i
 \end{verbatim}

\end{example}

All the entries of the matrix $U$ are positive real numbers. Thus, the Perron-Frobenius theorem says that it has a 
unique eigenvalue with maximal absolute value that we simply denote as Perron-Frobenius eigenvalue. 
Using Theorem \ref{pole} we can deduce the complexity and the Frobenius
complexity of the Frobenius algebra. This was already observed by F.~Enescu and Y.~Yao in \cite{EY15}.

\begin{corollary} \label{perron}
Let $\lambda$ be the Perron-Frobenius eigenvalue of the matrix $U$ associated to the twisted ring $A=T( V_r (R[x_1,\dots , x_m]))$.
Then $\cx(A)=\lambda.$

\end{corollary}

\vskip 2mm

We end this section with some computations for the twisted Veronese subring $A=T( V_r (R[x_1,\dots , x_{m}]))$ with 
$R$ being of characteristic $p$. F.~Enescu and Y.~Yao computed explicitly the case $(m,r,p)=(4,1,2)$ in \cite{EY16} 
and $(m,r,p)=(5,2,3)$ in \cite{EY15}. In the following tables we present,  for some small cases, the minimal polynomial $m(T)$, the complexity
and the Frobenius complexity of $A$ varying the characteristic of the ring. We point out that
$\lim_{p\to \infty} \cx_F(A)=m-1$ (see \cite[Theorem 4.1]{EY15}).

\vskip 2mm

\def\arraystretch{1.3}
\begin{table}[!ht] \label{table:T2}
\centering
\begin{tabular}{|c|c|c|c|}
\hline
\(p\) & \( m(T)\) & \(\cx(A) \) & \(\cx_F(A) \)\\[1.5pt] \hline
\hline
\(2\) & \(T^3 - 25T^2 + 165T - 280\) & \( 15.5436 \)  & \( 3.9583 \) \\[1.5pt] \hline
\(3\) & \(T^3 - 105T^2 + 2205T - 8505\) & \( 78.1909 \)  & \( 3.9679 \) \\[1.5pt] \hline
\(5\) & \(T^3 - 710T^2 + 66625T - 687500\) & \( 601.0565 \)  & \( 3.9757 \) \\[1.5pt] \hline
\(7\) & \(T^3 - 2590T^2 + 660275T - 12941390\) & \( 2306.1190 \)  & \( 3.9793 \) \\[1.5pt] \hline
\(11\) & \(T^3 - 15103T^2 + 14857953T - 692680351\) & \( 14048.9228 \)  & \( 3.9828 \) \\[1.5pt] \hline
\(13\) & \(T^3 - 29120T^2 + 47246485T - 3040889670\) & \( 27399.7078 \)  & \( 3.9838 \) \\[1.5pt] \hline
\end{tabular}
\vspace{5pt}
\caption{ The case $(m,r)=(5,1)$ .}
\end{table}

\def\arraystretch{1.3}
\begin{table}[!ht] \label{table:T2}
\centering
\begin{tabular}{|c|c|c|c|}
\hline
\(p\) & \( m(T)\) & \(\cx(A) \) & \(\cx_F(A) \)\\[1.5pt] \hline
\hline
\(2\) & \(T^2 - 15T + 40\) & \( 11.5311 \)  & \( 3.5275 \) \\[1.5pt] \hline
\(3\) & \(T^2 - 60T + 420\) & \( 51.9089 \)  & \( 3.5950 \) \\[1.5pt] \hline
\(5\) & \(T^2 - 390T + 9625\) & \( 363.5230 \)  & \( 3.6633 \) \\[1.5pt] \hline
\(7\) & \(T^2 - 1400T + 82320\) & \( 1338.4982 \)  & \( 3.6997 \) \\[1.5pt] \hline
\(11\) & \(T^2 - 8052T + 1591876\) & \( 7849.1923 \)  & \( 3.7400 \) \\[1.5pt] \hline
\(13\) & \(T^2 - 15470T + 4844385\) & \( 15150.2437 \)  & \( 3.7528 \) \\[1.5pt] \hline
\end{tabular}
\vspace{5pt}
\caption{ The case $(m,r)=(5,2)$ .}
\end{table}

\vskip 2mm

In order to describe the generating function we only have to follow the arguments developed in the proof of Theorem \ref{recurrence}.
For example, in the case that $(m,r)=(5,1)$ and $p=5$ we have:

$$ \cG_A(T)= 1 +70 T + T^2 \left( \frac{15575 - 2913750 T +38359375 T^2}{1 - 710T + 66625T^2 - 687500T^3}  \right)$$

\section{Open questions}

Let $R$ be a commutative Noetherian ring and let $A= \oplus_{e\geq 0} A_e$ be a 
(non-necessarily commutative) $\mathbb{N}$-graded ring. When dealing with the generating 
function of $A$  and its relation with the complexity, there are some questions that immediately come to mind. 
The obvious one is:

\begin{itemize}
 \item When is the generating function of $A$ a rational function?
 \end{itemize}
 
 \noindent In this case, as a consequence of Theorem \ref{pole}, we have that the complexity of $A$ is finite.
 A priori, rationality of the generating function seems to be a stronger condition that involves the 
 structure of the complexity sequence  rather than its asymptotic behaviour. Then we may ask: 
 
\begin{itemize} 
 \item Is there any  $\mathbb{N}$-graded ring $A$ with finite complexity but non-rational generating function?
\end{itemize}

Assume that the generating sequence of $A$ is rational, then we may consider a generalization of the complexity $\cx(A)$
just taking into account all the eigenvalues of the characteristic polynomial associated to the corresponding linear recurrence ordered
by their absolute values. Namely,  we may consider the {\it complexity eigenvalues} of $A$
$$\cx_{\rm eig}(A)=(|\lambda_1|, |\lambda_2|, \dots , |\lambda_k|),$$ 
where the largest eigenvalue corresponds to the complexity (see Corollary \ref{linear_cx}).
In the case that $R$ is a local complete commutative ring of positive characteristic $p>0$, we may also
consider the {\it Frobenius complexity eigenvalues}
$$\cx_{F,\rm eig}(R):= \log_p(\cx_{\rm eig}(\cF(E_R))$$
Then we may ask:
\begin{itemize} 
 \item Does $\lim_{p\to \infty} \cx_{F,\rm eig}(R)$ exist?
 \item What is this limit for the twisted Veronese subring $A=T( V_r (R[x_1,\dots , x_{m}]))$ ?
\end{itemize}

\noindent Numerical experimentation with the twisted Veronese subring suggested  the following question:
\begin{itemize} 
 \item Are all the eigenvalues of the matrix $U$ different so $U$ diagonalizes?
\end{itemize}
Finally, it would be really useful to have a larger set of examples, not only  new examples of Frobenius algebras but other possible
interesting objects such as rings of differential operators over non-regular rings in any characteristic.


\end{document}